\documentclass[a4paper]{article}
\usepackage{amsthm,amssymb,amsmath,enumerate,graphicx,psfrag,color}
\usepackage{algorithm}
\usepackage{tikz}

\usetikzlibrary{backgrounds}

\newtheorem{definition}{Definition}

\newtheorem{theorem}[definition]{Theorem}

\newtheorem{lemma}[definition]{Lemma}
\newtheorem{conjecture}[definition]{Conjecture}

\newcommand{\bigO}{\ensuremath{O}}

\newcommand{\comment}[1]{}

\newcommand{\es}{\emptyset}

\newcommand{\cK}{\mathcal{K}}
\newcommand{\cL}{\mathcal{L}}

\newcommand{\sm}{\setminus}

\title{Long cycles through prescribed vertices have the Erd\H{o}s-P\'osa property}
\author{Henning Bruhn, Felix Joos and Oliver Schaudt}
\date{}

\begin{document}
\maketitle

\begin{abstract}
\noindent
We prove that for every graph, any vertex subset $S$, and given integers $k,\ell$:
there are $k$ disjoint cycles of length at least $\ell$ that each contain at least one 
vertex from $S$, or a vertex set of size $\bigO(\ell \cdot k \log k)$ that meets all such cycles.
This generalises 
previous results of Fiorini and Herinckx and of Pontecorvi and Wollan.

In addition, we  describe an algorithm for our main result that
runs in $\bigO(k \log k \cdot s^2 \cdot (f(\ell) \cdot n+m))$ 
time, where $s$ denotes the cardinality of $S$.
\end{abstract}

\section{Introduction}

Menger's theorem is an example of a most satisfactory type of existence result: 
either there are $k$ objects of the desired kind (disjoint paths between two vertex sets),
or there is a simple obstruction that excludes their existence (a separator of 
less than $k$ vertices). In many other situations, however, such an ideal characterisation 
cannot be achieved. The classic theorem of Erd\H os and P\'osa
is, in that respect, the next best type of existence result: there are $k$ disjoint
cycles, unless there is a simple obstruction that excludes the existence of 
\emph{many} more than $k$ disjoint cycles.

\begin{theorem}[Erd\H os and P\'osa~\cite{EP62}]\label{epthm}
Any graph either contains $k$ disjoint cycles or there is 
vertex set of size $\bigO(k\log k)$ meeting all cycles. 
\end{theorem}

The result of Erd\H os and P\'osa was the starting point for a series of 
articles that adapted the theorem to new settings or generalised it in various
directions. Two of these directions are the extension
to long cycles and to cycles through a specific vertex set.

\begin{theorem}[Fiorini and Herinckx~\cite{FH13}]\label{Birmelthm}
For any graph and any integer $\ell$, the graph either contains $k$ disjoint
cycles of length at least $\ell$  or a vertex set 
of size $\bigO(\ell\cdot k\log k)$
that meets all such cycles.\sloppy
\end{theorem}

That there is always such a \emph{hitting set}, a vertex set  meeting all 
cycles of length at least $\ell$, of a size depending only on $k$ and $\ell$
is a consequence of a more general result by Robertson and Seymour~\cite{RS86}. 
The bound on the hitting set was subsequently improved by Thomassen~\cite{Tho88}, 
followed by Birmel{\'e}, Bondy and Reed~\cite{BBR07}, until Fiorini and Herinckx~\cite{FH13} 
established the currently best bound stated above.

Kakimura, Kawarabayashi and Marx~\cite{KKM11} were the first to extend the Erd\H os-P\'osa
theorem to \emph{$S$-cycles}, the cycles in a graph that each contain a vertex
from a given vertex set $S$. 
The bound on the hitting set in the theorem below is due to
Pontecorvi and Wollan~\cite{PW12}.

\begin{theorem}[Kakimura et al.~\cite{KKM11}, and Pontecorvi and Wollan~\cite{PW12}]\label{paulsthm}
For any graph and any vertex subset $S$, the graph either contains $k$ disjoint
$S$-cycles  or a vertex set of size $\bigO(k\log k)$ that meets all  $S$-cycles.\sloppy
\end{theorem} 

In this article we bring these two divergent directions together. 
That is, we prove the following theorem, which extends both Theorems~\ref{Birmelthm}
and~\ref{paulsthm}.

\begin{theorem}\label{thm:main}
Let $k$ and $\ell$ be integers. For any graph $G$ and any subset of vertices $S$ one of the following holds:
\begin{enumerate}[\rm (a)]
	\item there exist $k$ vertex-disjoint $S$-cycles of length at least $\ell$, or
	\item there is a set $X$ with $|X| = \bigO(\ell \cdot k \log k)$ such that $G-X$ does not contain any $S$-cycle of length at least $\ell$.
\end{enumerate} 
\end{theorem}

Pontecorvi and Wollan  describe an $\bigO(mn)$-time algorithm 
that returns one of the two possible outcomes of their Theorem~\ref{paulsthm}. 
Our proof is also of algorithmic nature.

\begin{theorem}\label{thm:algo}
Let $f(\ell) = 2^{2\ell} (2\ell)!$ for positive integers $\ell$.
There is an algorithm that, on input of a graph $G$, a vertex subset $S$ of size $s$ and
integers $k$ and $\ell$, computes in time 
$\bigO(k \log k \cdot s^2 \cdot (f(\ell) \cdot n+m))$
one of the two outcomes {\rm (a), (b)} of Theorem~\ref{thm:main}.
\end{theorem}
Note that our algorithm runs in FPT time when parameterized by $\ell$.
In fact, the factor $f(\ell)$ comes from the subroutine of finding cycles of length at least $\ell$.
Since this problem is NP-hard in general, we cannot expect a running time which is polynomial in $\ell$.


\medskip
We briefly discuss some of the research initiated by the Erd\H os-P\'osa theorem.
A family $\mathcal H$ of graphs is said to have the \emph{Erd\H{o}s-P\'osa property}
if there is a function $f:\mathbb N\to\mathbb N$ so that 
any graph either contains $k$ disjoint subgraphs that are isomorphic to graphs in $\mathcal H$,
or if it contains a  vertex set of size $f(k)$ meeting all such subgraphs.
Long cycles and, stretching the definition a bit, $S$-cycles are just two of many 
examples having the Erd\H{o}s-P\'osa property. Others include:
\begin{itemize}
	\item the family of cycles of length $0$~mod~$m$ for any integer $m\geq2$\\ (Thomassen~\cite{Tho88}),\vspace{-2mm}
	\item the family of cycles of length not equal to $0$~mod~$m$ for any odd integer  $m\geq 3$ (Wollan~\cite{Wol11}),\vspace{-2mm}
	\item the family of graphs that can be contracted to a specific planar graph (Robertson and Seymour~\cite{RS86}),\vspace{-2mm}
	\item and the family of all (directed) cycles in a digraph (Reed et al.~and Havet and Maia \cite{RRST96,HM13}).
\end{itemize}

Other natural classes of graphs, in contrast, fail to have the Erd\H{o}s-P\'osa property:
for example, the family of odd cycles, 
clique minors (graphs that can be contracted to a given complete graph $K_p$ with $p\geq 5$),
and the family of cycles of length $\ell$~mod~$m$ for any $\ell\not=0$ and even $m$.
Somewhat surprisingly,
this changes if high connectivity is imposed.
Indeed, in highly connected graphs, odd cycles do have the Erd\H{o}s-P\'osa property,
see Thomassen \cite{Tho01}, Rautenbach and Reed \cite{RR01}, and Kawarabayashi and Wollan \cite{KW06}; as do clique minors, see
Diestel et al.~\cite{DKW12}; and as do cycles with arbitrary modularity constraints, see
Kawarabayashi and Wollan \cite{KW06}.

Coming back to $S$-cycles, we note that there is a long-standing interest in 
cycles through a prescribed set of vertices.
Probably the best known result is due to Dirac \cite{Dir60} who proved that in every $k$-connected graph ($k\geq 2$),
there is a cycle containing any given set of $k$ vertices.
Bondy and Lov{\'a}sz \cite{BL81} investigated this further and proved 
that every non-bipartite $k$-connected graph ($k\geq 2$) has an odd cycle containing any set of $k-1$ vertices
and every $k$-connected graph ($k\geq 3$) has an even cycle containing any set of $k$ vertices.

Just as odd cycles, odd $S$-cycles do not have the Erd\H{o}s-P\'osa property
in general, but gain it in highly connected graphs; see~\cite{Joo14}.
For cycles in digraphs the situation is slightly different
as demonstrated by an example of Wollan (see Kakimura and Kawarabayashi \cite{KK12}):
while (directed) cycles 
in digraphs have the  Erd\H{o}s-P\'osa property, the property is lost
when cycles are replaced by $S$-cycles. Whether high connectivity restores the property
appears to be unknown.

\section{Preliminaries and short discussion}

We use standard graph theory notation as found in Diestel~\cite{Die10}. 

\medskip
The best known proof of the Erd\H os-P\'osa theorem is certainly due to Simonovits~\cite{Sim67}.
Indeed, both later proofs of the Theorems~\ref{Birmelthm} and~\ref{paulsthm}
rely on refinements of Simonovits' strategy. We will follow it as well. 

In his proof
Simonovits grows step by step
a subgraph $H$ of the graph $G$ that encapsulates at  the same time
a candidate hitting  as well 
as a set of disjoint cycles. 
The graph $H$ is a subdivision of a cubic multigraph, and it 
turns out that either $H$
has many vertices of degree~$3$, 
in which case there are many cycles, or there are few of them, which means
they may play the role of hitting set. 

That any such $H$ with many vertices of degree~$3$ yields many disjoint cycles
is due to the theorem below.
For an integer $k \ge 2$ let
\begin{align*}
	s_k= 4k(\log k + \log \log k +4),
\end{align*}
while we put $s_k=1$ for $k=1$. (The logarithm is base~$2$.)

\begin{theorem}[Simonovits~\cite{Sim67}]\label{thm:simonovits}
Every cubic multigraph with at least $s_k$ many vertices contains $k$ disjoint cycles.
\end{theorem}
We note that the proof of the theorem can be turned into 
an algorithm that runs in $\bigO(n)$-time.

For the proof of our main result we naturally borrow some arguments from Fiorini and Herinckx~\cite{FH13},
from Pontecorvi and Wollan~\cite{PW12}.
In particular, both pairs of authors, Fiorini and Herinckx, and Pontecorvi and Wollan
adapt Simonovits' graph $H$ so that it only contains cycles of the desired kind, 
that is, either long cycles or $S$-cycles. We will do the same and force $H$ to contain
only long $S$-cycles. Ensuring that this is still the case when we grow $H$ takes up 
the main effort of the proof. 

The inductive proof of
Fiorini and Herinckx relies on a result of 
 Birmel{\'e} et al.~\cite{BBR07}\footnote{
The main result of Birmel{\'e} et al.~in~\cite{BBR07} says that every graph without $k$ long cycles contains
a hitting set for long cycles of size $O(\ell k^2)$ proved by induction on $k$.}
for its base case: that every graph without two disjoint long cycles has 
a hitting set of size at most $2\ell+3$.
Strikingly, this case turns out to be the longest and most involved 
part in the argumentation of~\cite{BBR07}.\footnote{
While the result was recently improved by Meierling et al.~\cite{MRS14},
the proof still takes a substantial effort.
}
Moreover, the non-constructive nature of the proof makes it difficult 
to extract an algorithm from it. 

Rather than extending the proof of Birmel{\'e} et al.\ to long $S$-cycles,
we avoid this somewhat complicated part completely. In that way, we not 
only can present a shorter and simpler proof but also directly obtain an 
algorithm, with only a little extra work.

\section{The proofs}

In this section we assume that $G$ is a graph and $S$ is a subset of the vertices of $G$.
Moreover, let $k,\ell$ be positive integers.

For a vertex set (interpreted as a graph without edges) or subgraph $H$ of some graph, we call
a path an \emph{$H$-path} 
if the endvertices of the path are contained in $H$, while 
all internal vertices lie outside~$H$
and it contains at least one edge not belonging to $H$.
In particular, an $H$-path contains at least one edge.

We call a cycle (in $G$)  \emph{long} if its length is at least $\ell$.
Let $H$ be a subgraph of $G$.
A set $X\subseteq V(H)$ is \emph{wide} if any path in $H$ with first and last vertex in $H-X$
that contains a vertex of $X$ has length at least $\ell$.

A subgraph $H$ of $G$ is a \emph{frame} if 
\begin{itemize}
\item every vertex of $H$ has degree~$2$ or~$3$ in $H$; and
\item every cycle contained in $H$ is a long $S$-cycle.  
\end{itemize}
Any vertex of degree~$3$ of $H$ is a \emph{branch vertex}, and we usually denote the
set of branch vertices by $B$.

Step by step, we 
will make our frame larger. Here is a simple way to do just that.
\begin{lemma}\label{lem:augment}
Let $H$ be a frame, and let $X\subseteq V(H)$ be wide and containing
all branch vertices of $H$. 
Consider an $H$-path $P$ of $G-X$ that links two components of $H-X$.
If every cycle in $H\cup P$ (that passes through $P$) is an $S$-cycle, 
then $H\cup P$ is again a frame with more branch vertices than $H$.
\end{lemma}
\begin{proof}
Let $H'=H\cup P$.
Since $X$ is wide, every cycle in $H'$ is long.
Moreover, by assumption, every cycle in $H'$ is an $S$-cycle.
Finally, observe that $H'$ satisfies the degree condition because the branch vertices
of $H$ are contained in $X$.
\end{proof}

A frame $H$ might have \emph{pendant cycles}; that is, a set $\mathcal K$
of pairwise disjoint long $S$-cycles that each meet $H$ in precisely one vertex.

For a tuple $(H,\mathcal K)$ of a frame together with a set of pendant 
cycles we define its \emph{score} as the tuple $(|B|,|S\cap V(H)|+|\mathcal K|)$.
We order scores lexicographically, 
which means that $(H',\mathcal K')$ has larger score than $(H,\mathcal K)$ 
if either $H'$ has more branch vertices than $H$, or if they have the same 
number of branch vertices but the number of vertices in $S$ contained in $H'$
plus the number of cycles in $\mathcal K'$ is higher than for $(H,\mathcal K)$.

\begin{proof}[Proof of Theorem~\ref{thm:main}]

Inductively, we define pairs $(H,\mathcal K)$ 
of a frame $H$ together with a set $\mathcal K$ of pendant cycles
until we either find $k$ disjoint long $S$-cycles or a hitting set $X$ as in the theorem.
We start the construction with $(\emptyset,\emptyset)$. 

Now, assume such a pair $(H,\mathcal K)$ to be already constructed.
Let $\cL$ be the set of components of $H$ that are cycles, and let $B$ the 
set of branch vertices of $H$. We define a multigraph $\mathcal H$ 
on $B\cup \cL$ as vertex set
with edge set $\mathcal E$: let $\mathcal H-\cL$ be the cubic multigraph 
of which $H-\bigcup_{C\in \cL}C$ is a subdivision, and let each $C\in \cL$
be a loop of $\mathcal H$ that is incident with itself, seen as a vertex. 
Thus, any edge $P\in\mathcal E$ of $\mathcal H$ is either a $B$-path in $H$
or a cycle component of $H$.

We bound the size of $\mathcal H$. 
By Theorem~\ref{thm:simonovits}, there are $k$ disjoint long $S$-cycles in $G$
if $|B|\geq s_{k-|\cL|}$.
As $|\mathcal E|=\tfrac{3}{2}|B|+|\cL|$
and $s_k\geq s_{k-1}+ \frac{3}{2}$, 
we may assume that
\begin{equation}\label{sizeofSimo}
|B|< s_k\text{ and }|\mathcal E|< \frac{3}{2}s_{k-|\cL|}+|\cL|\leq \frac{3}{2}s_{k}.
\end{equation}
If $\mathcal K$ consists of at least $k$ cycles, then again we can obviously stop
as we require the cycles in $\mathcal K$ to be pairwise disjoint long $S$-cycles. 
So we may assume that
\begin{equation}\label{sizeofpendant}
|\mathcal K|< k.
\end{equation}
These estimations give us an upper bound on the score:
\begin{equation}\label{scorebound}
\text{the score of $(H,\mathcal K)$ is less than $(s_k,|S|+k)$.}
\end{equation}

Next, we define a wide vertex set $X$ that is a candidate  for the hitting set
sought for in the theorem. The set $X$ comprises three types of subsets, namely sets $X_b$
for every branch vertex $b$ of $H$, sets $X_P$ for every edge $P\in\mathcal E$ of $\mathcal E$
and finally sets $X_K$ for every pendant cycle $K\in\mathcal K$. We put
\begin{equation}\label{eqn:defX}
	X= \bigcup_{b\in B}X_b \cup \bigcup_{P \in \mathcal E}X_P \cup \bigcup_{K \in \mathcal K} X_K.
\end{equation}
An illustration of the different types making up $X$ is given in Figure~ \ref{Xfig}.

\begin{figure}[ht]
\centering
\begin{tikzpicture}[every edge quotes/.style={},scale=1]
\tikzstyle{hvertex}=[thick,circle,inner sep=0.cm, minimum size=2.5mm, fill=white, draw=black]
\tikzstyle{hedge}=[ultra thick]
\tikzstyle{Xset}=[line width=10pt,line cap=round, gray]
\tikzstyle{svertex}=[thick,circle,inner sep=0.cm, minimum size=1.5mm, fill=black, draw=black]

\node[hvertex,label=above right:$X_b$] (b1) at (0,1){};
\node[hvertex,label=above:$X_b$] (b2) at (9,0.5){};


\path (b1) -- +(0.9,-0.2) coordinate (bb1){};
\path (b2) -- +(-0.9,0.2) coordinate (bb2){};

\draw[hedge] (b1) -- +(-0.8,0.8)  node[midway] (x1){};
\draw[hedge] (b1) -- +(-0.9,-0.6) node[midway] (y1){};
\draw[hedge] (b1) -- (bb1) node[midway] (z1){};

\draw[hedge] (b2) -- +(0.9,0.3)  node[midway] (x2){};
\draw[hedge] (b2) -- +(0.3,-0.8) node[midway] (y2){};
\draw[hedge] (b2) -- (bb2) node[midway] (z2){};

\draw[hedge,rounded corners] (bb1) -- ++(0.8,0.1) -- ++(0.5,-0.4) 
node[near start] (ls1){} node[svertex,label=above right:$X_P$] (s1){} 
-- ++(1.1,-0.1) node[midway] (rs1){} 
-- ++(1.3,0.4) node[svertex,midway] {} -- ++(0.7,-0.1) node[near start] (lK){} 
node[hvertex,minimum size=1.5mm,label=above:$X_K$] (yK){} 
-- ++(1,0.3) node[midway] (rK){} -- ++(0.3,-0.2) node[svertex] {} 
-- ++(1,-0.2) node[midway] (ls2){} node[svertex,label=above:$X_P$] (s2){}  
-- (bb2) node[near end] (rs2){} ;

\draw[very thick] (yK) .. controls ++(-1,-1.5) and ++(1,-1.5) .. (yK) node[near start,svertex] {}; 

\begin{scope}[on background layer]
\draw[Xset] (b1) -- (x1);
\draw[Xset] (b1) -- (y1);
\draw[Xset] (b1) -- (z1);
\draw[Xset] (b2) -- (x2);
\draw[Xset] (b2) -- (y2);
\draw[Xset] (b2) -- (z2);

\draw[Xset] (rK) -- (yK) -- (lK);
\draw[Xset] (ls1) -- (s1) -- (rs1);
\draw[Xset] (ls2) -- (s2) -- (rs2);
\end{scope}

\end{tikzpicture}
\caption{Definition of $X$; vertices of $S$ in black}\label{Xfig}
\end{figure}
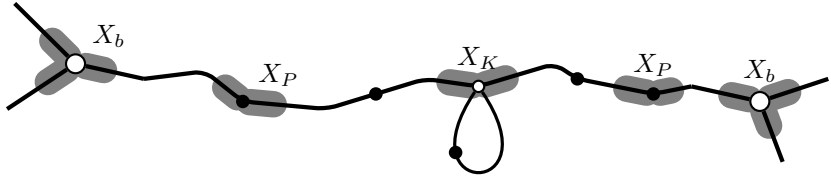

We define the different types, beginning with the branch vertices. 
For every branch vertex $b\in B$ of $H$, 
we let  $X_b$ be the set of vertices of distance at most $\tfrac{\ell-1}{2}$ to $b$ in $H$.
We note for later use that
\begin{equation}\label{sizeofXb}
\left|\bigcup_{b\in B} X_b\right| \le \tfrac{3\ell}{2}|B| \mbox{ and } B\subseteq X.
\end{equation}

Next, for each $P\in\mathcal E$ we define a set $X_P \subseteq V(H)$. If $P$ 
is disjoint from $S$ then we simply put $X_P=\emptyset$. 
If $P$ is a path that contains some vertex from $S$, we 
let $s_P$ be the first and $s'_P$ be the last vertex of $S$ in $P$.
In this case, we choose $X_P$ to be the set of  vertices of $P$ 
of distance at most $\frac{\ell-1}{2}$ from $\{s_P,s'_P\}$ in $P$.
Finally, if $P\in \cL$,
that is, if $P$ is a cycle component of $H$, 
then it has to contain a vertex of $S$, since every cycle of $H$ 
is an $S$-cycle. We pick some vertex $s_P$ and let $X_P$ again be the set
vertices of $P$ 
of distance at most $\frac{\ell-1}{2}$ from $s_P$ in $P$.
We note that in any of the cases $|X_P|\leq 2 \ell$.

Finally, for each $K\in\mathcal K$ there is, by definition, a unique vertex $y_K$
shared by $K$ and $H$. We define $X_K$ to be the set of vertices in $H$ of distance 
at most $\ell-1$ to $y_K$. Note that $|X_K\sm \bigcup_{b \in B} X_b|\leq 2\ell$.
We  observe that
\begin{equation}\label{sizeofX}
|X|\leq \frac{3}{2}\ell\cdot s_k+ 2\ell\cdot \frac{3}{2}s_k+ 2\ell\cdot k.
\end{equation}
This estimation follows from~\eqref{sizeofSimo},~\eqref{sizeofpendant}, and~\eqref{sizeofXb}. 
Note, moreover, that 
\begin{equation}\label{wideX}
\text{
$X$ and $X\sm \{y_K\}$, for all $K\in\mathcal K$, are wide sets.
}
\end{equation}

Having defined $X$, we observe that there are two possibilities. Either $G-X$
is devoid of long $S$-cycles, in which case we are done, or there is still
such a cycle. In that case, which will occupy the rest of the proof, 
we will  change $(H,\mathcal K)$ into a frame-pendant cycles 
pair $(H',\mathcal K')$ of higher score. The score, however, is bounded from above,
by~\eqref{scorebound}, which means that this procedure eventually
ends.

So let us consider a long $S$-cycle $C$ in $G-X$. We distinguish three cases, depending
on whether $C$ meets $H$ nowhere, in one vertex or in more vertices. 

The easiest case is when $C$ is disjoint from $H$.
If $C$ is in addition disjoint from the cycles in $\mathcal K$, then $H\cup C$ is a frame and 
$(H\cup C,\mathcal K)$
has higher score than $(H,\mathcal K)$, as the former contains more vertices from $S$.
Suppose $C$ intersects a cycle $K\in\mathcal K$.
Let $P$ be a path in $K$ joining $C$ and $y_K$.
Now $(H\cup P\cup C,\es)$ has a higher score than $(H,\mathcal K)$ as it contains two more branch vertices.
Note that there is no cycle in $H\cup P\cup C$ containing an internal vertex of $P$
and hence $H\cup P\cup C$ is a frame.

Next, assume that $C$ meets $H$ in precisely one vertex $y_C$. If $C$ is disjoint from 
any cycle in $\mathcal K$, we can add $C$ to $\mathcal K$. 
As then $(H,\mathcal K\cup\{C\})$ has higher score than $(H,\mathcal K)$
we are done again. Thus, assume there is some $K\in\mathcal K$ having 
a vertex with $C$ in common. Note that this cannot be the unique vertex $y_K$
of $K$ in $H$ as $X$ contains all such vertices. 

Pick some vertex $s\in (C\cup K)\cap S$, which exists as both are $S$-cycles. 
Then there exists an $H$-path $Q$ through $s$ in $C\cup K$: if $s\in \{y_C,y_K\}$
then any path in $C\cup K$ from $y_C$ to $y_K$ will do as $Q$; if $s\notin \{y_C,y_K\}$,
on the other hand, then $s$ cannot, in $C\cup K$, be separated by single vertex from $\{y_C,y_K\}$,
which means there is such a $Q$.

Applying Lemma~\ref{lem:augment}
to $Q$ and the set $X- y_K$, which is wide by~\eqref{wideX}, we see that $H\cup Q$
is a frame with more branch vertices than $H$. Thus $(H\cup Q,\emptyset)$
has higher score than $(H,\mathcal K)$.
This finishes the case of a unique
common vertex of $H$ and $C$. 

We turn to the remaining case: $C$ meets $H$ in at least two vertices. 
Since $C$ is an $S$-cycle, $C$ contains an $H$-path $Q^*$ 
through a vertex of $S$ or there is a vertex $s^*\in V(H)\cap V(C)$ such that the two neighbours of $s^*$ in $C$ coincide with two neighbours of $s^*$ in $H$.
In the latter case, we denote the trivial path starting and ending in $s^*$ also by $Q^*$.

If $Q^*$ links two components of $H-X$ then we apply 
Lemma~\ref{lem:augment} again to see that $H\cup Q^*$ is a frame with more 
branch vertices than $H$. Consequently, $(H\cup Q^*,\emptyset)$ has higher 
score than $(H,\mathcal K)$ and we are done. 

Therefore, $Q^*$ meets a single component $D$ of $H-X$.
Since all branch vertices of $H$ are contained in $X$, 
the component $D$ of $H-X$ is a subset of some $P^*\in\mathcal E$. 
As an $H$-path, $Q^*$ meets $H$ and thus $P^*$ in precisely its endvertices;
let these be $q_1^*,q_2^*$
(if $Q^*$ is not an $H$-path but a trivial path, then let $q_1^*=q_2^*=s^*$).

Suppose first that there is a $K\in \cK$ such that $Q^*$ and $K$ intersect.
Recall that $y_K\in X$ and hence $y_K \notin V(Q^*)$.
Picking any $s\in S$ in the $S$-cycle $K$, we see that, in $Q^*\cup K$,  
there is an $H$-path $P$ through $s$
starting in one of $\{q_1^*,q_2^*\}$ and ending in $y_K$.
Since $X-y_K$ is wide, we conclude by Lemma \ref{lem:augment} that $H\cup P$ is a frame and
$(H\cup P,\emptyset)$ has a higher score than $(H,\cK)$.

Hence, we may assume that $Q^*$ meets no cycle in $\cK$.
If $q_1^*P^*q_2^*$ does not contain an element of $S$,
then let $H'$ arise from $H$ by replacing the path $q_1^*P^*q_2^*$ by $Q^*$,
which results in $(H',\cK)$ having a higher score than $(H,\cK)$.

Therefore, we may assume that $D\cap S\not=\es$.
In particular, $P^*$ contains at least three vertices from $S$ and hence $D$ lies in $P^*$ within $s_{P^*}$ and $s_{P^*}'$.
If $C$ 
contains vertices from $H-D$, then $C$ also contains an $H$-path $Q$
from $H-D$ to $D$. 
Observe that any cycle in $H\cup Q$ that passes through $Q$ contains a vertex from $\{s_{P^*},s_{P^*}'\}\subset S$, 
which in turn 
lets us apply Lemma~\ref{lem:augment} again to deduce that $H\cup Q$ 
is a frame with more branch vertices. Since 
$(H\cup Q,\emptyset)$ has higher 
score than $(H,\mathcal K)$, we may assume that $C$ meets $H$ only in $D$.

Let $r_1$ and $r_2$ be the first and the last vertex of $P^*$ belonging to $C$.
In $H$, 
replace the subpath $r_1P^*r_2$ by $C$ in order to obtain a graph $H'$ of minimal
degree~$2$ and maximal degree~$3$. Moreover, as $X$ is wide, any cycle in $H'$ is long. 
Let us check that all cycles in $H'$ are $S$-cycles. This is clearly the case for $C$
and for any cycle that avoids $C$. 

Any cycle that meets $C$, other than $C$ itself, 
 contains both of $s_{P^*}$ and $s_{P^*}'$ and is thus an $S$-cycle. 
Consequently, $H'$ is a frame.
However, $H'$ has two more branch vertices, namely $r_1,r_2$,
than $H$. Again $(H',\emptyset)$ has higher score than $(H,\mathcal K)$.
\end{proof}

Before we sketch how the proof  can be turned into an algorithm let us note that 
the hitting set has size at most
\begin{equation}\label{eqn:truebound}
|X| \le \frac{9}{2}\ell s_k+2\ell k = 18\ell k (\log k +\log\log k +37/9).
\end{equation}
While, with a bit of effort, this bound can certainly be improved somewhat,  we did 
not see how to lower it substantially.  

\begin{proof}[Proof of Theorem~\ref{thm:algo}]
If $s \le k$, we  simply output $X=S$ as 
the removal of $S$ obviously destroys all  $S$-cycles of $G$, long or not.
Thus we may assume $k \le s$.

Following the steps of the proof of Theorem~\ref{thm:main} 
we start with the frame-pendant cycles pair $(\emptyset,\emptyset)$.
In each iteration of the algorithm, we improve this pair, measured by its score.
Since the bound~\eqref{scorebound} will still be valid, the algorithm 
will perform at most $s_k (k + s) = \bigO(k \log k \cdot s)$ iterations (recall that $k\leq s$).

Assume that the algorithm has already constructed a pair $(H,\cK)$, 
and let $B$ be the set of branch vertices of $H$, and $\mathcal L$
the set of its cycle components. In Theorem~\ref{thm:main} we argued 
via Simonovits' Theorem~\ref{thm:simonovits} that $|B|\geq s_{k-|\mathcal L|}$
guarantees $k$ disjoint long $S$-cycles. As these cycles can be computed in $\bigO(n)$-time, 
we are done in that case. Similarly, the bound~\eqref{sizeofpendant} on $|\mathcal K|\leq k$
can also be assumed; otherwise we output the $k$ disjoint long $S$-cycles in $\mathcal K$.

We can compute the set $X$ as in~\eqref{eqn:defX} in $\bigO(\ell n)$ time, 
since $H$ has only $\bigO(n)$ many edges, while
the graph $G-X$ can be computed in $\bigO(m+n)$ time.

Next, we need to check whether there is still a long $S$-cycle in $G-X$. For this, 
we use an algorithm of Bodlaender~\cite[Thm.~5.3]{Bod93}:
it finds a long cycle through a prescribed vertex in any graph
or concludes that there is no such cycle.
We run this algorithm for each vertex of $S$, which amounts 
to $\bigO(s (2^{2\ell} (2\ell)! \cdot n + m))$ time in total.

Now, if there is no long $S$-cycle in $G-X$, we are done and output $X$.
Otherwise, Bodlaender's algorithm finds a long $S$-cycle, say $C$.
Since $H \cup C$ is of maximal degree~$4$, 
it has a linear number of edges, and we may thus check in $\bigO(n)$ time 
in which of the cases of Theorem~\ref{thm:main} we are in, 
and improve $(H,\cK)$ accordingly.

As, consequently, each iteration takes $\bigO(s (2^{2\ell} (2\ell)! \cdot n + m))$-time, 
the total running time amounts to
$\bigO(k \log k \cdot s^2 \cdot ( 2^{2\ell} (2\ell)! \cdot n+m))$.
\end{proof}

\section{Conclusion}

We conclude the article with some observations. 
Our contribution consists in a common generalisation of Theorems~\ref{Birmelthm} and~\ref{paulsthm}. 
Introducing weights would be another obvious
extension. More precisely, given weights $w:V(G)\to \mathbb R_+$ on 
the vertices of a graph $G$, we may ask whether there are $k$ disjoint cycles of 
weight at least $\ell$ each, or a hitting set $X$ of a total weight $w(X)$ bounded in $k$ and $\ell$. 

This proposal, however, still needs a small adaption. Indeed, we cannot expect to bound $w(X)$
only in $k$ and $\ell$ if we allow arbitrarily high weights. This can be seen be taking any 
graph that does not contain any $k$ disjoint cycles, for instance a triangle, and imposing
arbitrarily high weights on all the vertices. On the other hand, there is no sense in having 
weights larger than $\ell$: any cycle containing such a vertex has already a weight of at 
least $\ell$, and this does not change if we cap the weights at $\ell$. 

So a weighted version of our main result would be:

\begin{theorem}\label{thm:wmain}
Let $k$ and $\ell$ be integers. For any graph $G$, any weight function $w:V(G)\to [0,\ell]$ 
 and any subset of vertices $S$ one of the following holds:
\begin{enumerate}[\rm (a)]
	\item there exist $k$ vertex-disjoint $S$-cycles of weight at least $\ell$, or
	\item there is a set $X$ with weight $w(X) = \bigO(\ell \cdot k \log k)$ such that $G-X$ does not contain any $S$-cycle of weight at least $\ell$.
\end{enumerate} 
\end{theorem}

Indeed, the proof of the above result is similar to that of Theorem~\ref{thm:main}, which can be seen as unit weight version of Theorem~\ref{thm:wmain}.
The major difference is in the definition of the hitting set $X$, given a frame with pendant cycles, since we have to adapt the notion of a wide set to the weighted setting.
For example, instead of including all vertices of distance $\ell/2$ from a branch vertex $b$, we simply have to include all vertices reachable from $b$ by a path of total weight at most $\ell$.
The other necessary adaptions we leave to the interested reader.

\medskip Next, we discuss the size of the hitting set. 
Our bound on the hitting set coincides, up to a constant factor, with the one given by Fiorini and Herinckx~\cite{FH13}.
As they already discuss,
the bound $\bigO(\ell\cdot k\log k)$ 
is asymptotically tight, up to a constant factor, in $k$ resp.~$\ell$
if the other parameter is kept constant.

For fixed~$\ell$ this follows already from the original probabilistic construction of Erd\H os 
and P\'osa~\cite{EP62}, and also from the explicit constructions of Simonovits~\cite{Sim67}. 
Indeed, in both cases the graphs have girth about $\log n$, which means that every cycle is long
(provided $n$ is large enough).

For fixed~$k$, Fiorini and Herinckx give the example of 
the disjoint union of $k-1$ cliques on $2\ell-1$ vertices each. 
Obviously, there are no 
$k$ long cycles, yet $(k-1)\ell$ vertices need to be deleted to guarantee
that no long cycle remains. 
 
If both, $\ell$ and $k$, are allowed to grow it is not clear whether 
a size of $\bigO(\ell\cdot k\log k)$  for the hitting set is best possible. 
In our opinion the existing lower bounds, provided by the above examples, are more convincing. Moreover, our proof 
does seem a bit wasteful. Simonovits' theorem assumes that the cycles found in 
the cubic multigraph
have length about $\log k$; if the cycles were shorter, one could obtain more cycles 
from it. If, however, the cycles obtained from our frame already have length about $\log k$ 
in $\mathcal H$, that is, in the cubic multigraph, then they have length about $\tfrac{\ell}{2}\log k$
in $H$ and thus in $G$. This is because in our construction branch vertices in $H$ have a
distance of at least $\tfrac{\ell}{2}$ from each other. 
We conjecture:

\begin{conjecture}
For every graph $G$, for any subset of vertices $S$,
and for any positive integers $k,\ell$,
there is a set of $k$ disjoint $S$-cycles of length at least $\ell$
or a set $X$ of size $O(k(\ell+\log k))$ such that $G-X$ does not contain any $S$-cycle of length at least $\ell$.
\end{conjecture}

\medskip

There are several more open problems related to our results.
\begin{enumerate}
	\item Does the class of cycles that contain at least $p$ vertices in $S$ have the Erd\H{o}s-P\'osa property?
	The same question can also be formulated in a weighted version.
	\item Already Birmel{\'e} et al.~\cite{BBR07} asked for an edge-version for long cycles.
While such a version is well-known for the original Erd\H os-P\'osa result 
it is not obvious how to deduce it from the vertex-version.
On the other hand, by making good use of the set~$S$
Pontecorvi and Wollan~\cite{PW12} could prove an edge-analogue of their result 
via  an easy gadget construction.
Unfortunately, this construction breaks down in the case of long $S$-cycles 
(or even just long cycles).
\end{enumerate}

\bibliographystyle{amsplain}
\bibliography{erdosposa}

\providecommand{\bysame}{\leavevmode\hbox to3em{\hrulefill}\thinspace}
\providecommand{\MR}{\relax\ifhmode\unskip\space\fi MR }
\providecommand{\MRhref}[2]{%
  \href{http://www.ams.org/mathscinet-getitem?mr=#1}{#2}
}
\providecommand{\href}[2]{#2}
\begin{thebibliography}{10}

\bibitem{BBR07}
E.~Birmel{\'e}, J.A. Bondy, and B.~Reed, \emph{The {E}rd{\H o}s-{P}\'osa
  property for long circuits}, Combinatorica \textbf{27} (2007), 135--145.

\bibitem{Bod93}
H.L. Bodlaender, \emph{On linear time minor tests with depth-first search},
  J.~Algorithms \textbf{14} (1993), 1--23.

\bibitem{BL81}
J.A. Bondy and L.~Lov{\'a}sz, \emph{Cycles through specified vertices of a
  graph}, Combinatorica \textbf{1} (1981), 117--140.

\bibitem{Die10}
R.~Diestel, \emph{Graph theory}, fourth ed., Springer, Heidelberg, 2010.

\bibitem{DKW12}
R.~Diestel, K.~Kawarabayashi, and P.~Wollan, \emph{The {E}rd{\H o}s-{P}\'osa
  property for clique minors in highly connected graphs}, J. Combin. Theory
  Ser. B \textbf{102} (2012), 454--469.

\bibitem{Dir60}
G.A. Dirac, \emph{In abstrakten {G}raphen vorhandene vollst\"andige 4-{G}raphen
  und ihre {U}nterteilungen}, Math. Nachr. \textbf{22} (1960), 61--85.

\bibitem{EP62}
P.~Erd{\H{o}}s and L.~P{\'o}sa, \emph{On the maximal number of disjoint
  circuits of a graph}, Publ. Math. Debrecen \textbf{9} (1962), 3--12.

\bibitem{FH13}
S.~Fiorini and A.~Herinckx, \emph{A tighter {E}rd{\H o}s-{P\'o}sa function for
  long cycles}, J.~Graph Theory \textbf{77} (2013), 111--116.

\bibitem{HM13}
F.~Havet and A.K. Maia, \emph{On disjoint directed cycles with prescribed
  minimum lengths}, INRIA Research Report \textbf{RR-8286} (2013).

\bibitem{Joo14}
F.~Joos, \emph{Parity linkage and the {E}rd{\H o}s-{P}\'osa property of odd
  cycles through prescribed vertices in highly connected graphs},
  arXiv:1411.6554.

\bibitem{KK12}
N.~Kakimura and K.~Kawarabayashi, \emph{Packing directed circuits through
  prescribed vertices bounded fractionally}, SIAM J. Discrete Math. \textbf{26}
  (2012), 1121--1133.

\bibitem{KKM11}
N.~Kakimura, K.~Kawarabayashi, and D~Marx, \emph{Packing cycles through
  prescribed vertices}, J.~Combin.\ Theory (Series B) \textbf{101} (2011),
  378--381.

\bibitem{KW06}
K.~Kawarabayashi and P.~Wollan, \emph{Non-zero disjoint cycles in highly
  connected group labelled graphs}, J.~Combin.\ Theory (Series B) \textbf{96}
  (2006), 296--301.

\bibitem{MRS14}
D.~Meierling, D.~Rautenbach, and T.~Sasse, \emph{The for {E}rd{\H o}s-{P}\'osa
  property long curcuits}, J.~Graph Theory \textbf{77} (2014), 251--259.

\bibitem{PW12}
M.~Pontecorvi and P.~Wollan, \emph{Disjoint cycles intersecting a set of
  vertices}, J.~Combin.\ Theory (Series B) \textbf{102} (2012), 1134--1141.

\bibitem{RR01}
D.~Rautenbach and B.~Reed, \emph{The {E}rd{\H o}s-{P}\'osa property for odd
  cycles in highly connected graphs}, Combinatorica \textbf{21} (2001),
  267--278.

\bibitem{RRST96}
B.~Reed, N.~Robertson, P.~Seymour, and R.~Thomas, \emph{Packing directed
  circuits}, Combinatorica \textbf{16} (1996), 535--554.

\bibitem{RS86}
N.~Robertson and P.~Seymour, \emph{Graph minors. {V}. {E}xcluding a planar
  graph}, J.~Combin.\ Theory (Series B) \textbf{41} (1986), 92--114.

\bibitem{Sim67}
M.~Simonovits, \emph{A new proof and generalizations of a theorem of {E}rd\"os
  and {P}\'osa on graphs without $k+1$ independent circuits}, Acta
  Math.~Acad.~Sci.~Hungar. \textbf{18} (1967), 191--206.

\bibitem{Tho88}
C.~Thomassen, \emph{On the presence of disjoint subgraphs of a specified type},
  J. Graph Theory \textbf{12} (1988), 101--111.

\bibitem{Tho01}
\bysame, \emph{The {E}rd{\H o}s-{P}\'osa property for odd cycles in graphs of
  large connectivity}, Combinatorica \textbf{21} (2001), 321--333.

\bibitem{Wol11}
P.~Wollan, \emph{Packing cycles with modularity constraints}, Combinatorica
  \textbf{31} (2011), 95--126.

\end{thebibliography}

\vfill

\small
\vskip2mm plus 1fill
\noindent
Version \today{}
\bigbreak

\noindent
Henning Bruhn
{\tt <henning.bruhn@uni-ulm.de>}\\
Felix Joos
{\tt <felix.joos@uni-ulm.de>}\\
Institut f\"ur Optimierung und Operations Research\\
Universit\"at Ulm, Ulm\\
Germany\\

\noindent
Oliver Schaudt
{\tt <schaudto@uni-koeln.de>}\\
Institut f\"ur Informatik\\
Universit\"at zu K\"oln, K\"oln\\
Germany

\end{document}